\documentclass[11pt]{amsart}
\usepackage{amssymb}
\usepackage{amscd}
\usepackage{mathrsfs}
\usepackage{cases}
\usepackage{latexsym}
\usepackage{amsmath}
\usepackage{stmaryrd}
\usepackage{color}
\usepackage{amsmath,amssymb,amscd,bbm,amsthm,mathrsfs,dsfont}
\usepackage{graphicx}
\usepackage{fancyhdr}
\usepackage{amsxtra,ifthen}
\usepackage{verbatim}
\usepackage{latexsym,epsfig,amsmath}
\usepackage{hyperref}
\usepackage{cite}
\usepackage{enumitem}

\newtheorem{theorem}{Theorem}[section]
\newtheorem{lemma}[theorem]{Lemma}
\newtheorem{proposition}[theorem]{Proposition}
\newtheorem{corollary}[theorem]{Corollary}
\newtheorem{definition}[theorem]{Definition}
\newtheorem{remark}[theorem]{Remark}

\newtheorem{conjecture}[theorem]{Conjecture}

\begin{document}
\title[Affine-triangular automorphisms]{Dynamical degrees of affine-triangular automorphisms in dimension four}
\author{Enbo Shao, Xiaosong Sun}
\address{Enbo Shao: Foundation Department, Liaoning Institute of Science and Technology, Benxi, Liaoning, 117004, China}
\email{shaoenbo@lnist.edu.cn}
\address{Xiaosong Sun: School of Mathematics, Jilin University, Changchun, Jilin, 130012, China}
\email{sunxs@jlu.edu.cn}
\subjclass[2020]{37F10, 14R10}
\keywords{Dynamical degrees, polynomial automorphisms}	
	\begin{abstract}
		In this paper, let \( \mathbb{A}^n_k \) be the affine $n$-space over an arbitrary field $k$. We show that dynamical degrees of affine-triangular automorphisms of \( \mathbb{A}^4_k \) are algebraic integers of degree $\le4$. As a consequence, if $\mathrm{char}(k)\neq2$, then dynamical degrees of qudratic automorphisms of $\mathbb{A}_k^4$ are algebraic integers of degree $\le4$. We also explore some results in higher dimensions.
        These results partially give the affirmative answer to the Conjecture in \cite{dang2021spectral}.
	\end{abstract}   
    
    \maketitle	
    
	\section{Introduction}
    \subsection{Dynamical degree}
     Let \( k \) be an arbitrary field, and let \( \mathbb{A}^n_k \) denote the affine \( n \)-space over \( k \).
     For each integer \( n \geq 1 \), a polynomial map \( f \in \mathrm{End}(\mathbb{A}^n_k) \) is an endomorphism of affine space defined by a tuple of polynomials:
     \[f: (x_1, \dots, x_n) \mapsto \big(f_1(x_1, \dots, x_n), \dots, f_n(x_1, \dots, x_n)\big)\]
     where $f_i \in k[x_1, \dots, x_n]$.
     For simplicity, we write \( f = (f_1, \dots, f_n) \) and write \( \mathrm{Aut}(\mathbb{A}^n_k) \) for the group of polynomial automorphisms, that is, bijective endomorphisms with inverse in \( \mathrm{End}(\mathbb{A}^n_k) \).

     The \emph{degree} of a polynomial map \( f \in \mathrm{End}(\mathbb{A}^n_k) \) is defined as
     \[\deg(f) := \max \left\{ \deg(f_1), \dots, \deg(f_n) \right\}.\]
     To study the asymptotic behavior of iterations of such maps, people introduce the notion of \emph{dynamical degree}. For \( f \in \mathrm{End}(\mathbb{A}^n_k) \), the (first) dynamical degree is defined as
     \[\lambda(f) := \lim_{r \to \infty} \big( \deg(f^r) \big)^{1/r}.\]
     This limit follows from Fekete’s subadditivity lemma~\cite{fekete1923verteilung} and measures the exponential growth rate of degrees under iteration. It plays a central role in algebraic dynamics, where it serves as a fundamental invariant that characterizes the complexity of a map.

     In low dimensions, dynamical degrees are relatively well understood. For example, it is known that if \( f \in \mathrm{End}(\mathbb{A}^1_k) \) or \( f \in \mathrm{Aut}(\mathbb{A}^2_k) \), then \( \lambda(f) \in \mathbb{Z} \)~\cite{jung1942ganze,van1953polynomial,van2000polynomial}. Moreover, the set of dynamical degrees of quadratic endomorphisms of \( \mathbb{A}^2_{\mathbb{C}} \) is exactly
     \[\left\{ 1, \sqrt{2}, \frac{1+\sqrt{5}}{2}, 2 \right\},\]
     classified in~\cite{maegawa2001classification}. The same set also arises as the set of dynamical degrees of degree-2 automorphisms of \( \mathbb{A}^3_{\mathbb{C}} \)~\cite{guedj2004dynamics}. 

     In higher dimensions and for maps of higher degree, the structure of dynamical degrees becomes significantly more intricate. For instance, the set of dynamical degrees of degree-3 automorphisms of \( \mathbb{A}^3_k \) has been computed as
     \[
     \left\{ 1, \sqrt{2}, \frac{1+\sqrt{5}}{2}, \sqrt{3}, 2, \frac{1+\sqrt{13}}{2}, 1+\sqrt{2}, \sqrt{6}, \frac{1+\sqrt{17}}{2}, \frac{3+\sqrt{5}}{2}, 1+\sqrt{3}, 3 \right\}
     \]
     (cf.~\cite[Theorem 2]{Bla}).

     These values are not all integers, but they are all \emph{algebraic integers}. 
     This leads to a natural and fundamental question in the field purposed by Favre \cite{dang2021spectral}:

     \begin{conjecture}\label{404}
      The dynamical degree of any element in \( \mathrm{End}(\mathbb{A}^n_k) \) is an algebraic integer of degree at most \( n \).
     \end{conjecture}

     This conjecture holds in many known cases. When \( f \in \mathrm{End}(\mathbb{A}^1_k) \) or \( f \in \mathrm{Aut}(\mathbb{A}^2_k) \), the result is classical. The dynamical degree of any element in \( \mathrm{End}(\mathbb{A}^2_{\mathbb{C}}) \) is a quadratic integer, as shown in~\cite[Theorem A]{favre2007eigenvaluations}.

     In dimension three, a notable family of maps has been studied in detail, the so-called \emph{affine-triangular automorphisms} (see Definition~\ref{013}). Their dynamical degrees can be computed explicitly and turn out to be algebraic integers of degree at most 2.
     
     \begin{theorem}~\cite[Theorem 1]{blanc2022dynamical}
     For each integer $d\ge 2$, the set of dynamical degrees of all affine-triangular automorphisms of $\mathbb{A}^3_k$ of degree $\le d$ is equal to
     \[
	       \left.\left\{\frac{a+\sqrt{a^2+4bc}}{2}\right| (a,b,c)\in \mathbb{N}^3, a+b\le d, c\le d\right\}\setminus \{0\}.
     \]
     Moreover, for all $a,b,c\in\mathbb{N}$ such that $\lambda=\frac{a+\sqrt{a^2+4bc}}{2}\neq 0$, this value indeed occurs as a dynamical degree.
     The dynamical degree $\lambda$ is achieved by one of the automorphisms
     \[
	     (x_3 + x_1^ax_2^b,x_2+x_1^c,x_1) \text{ or }(x_3+x_1^ax_2^{bc},x_1,x_2) \, .
     \] 
     \end{theorem}
     
     More generally, if char$(k)=0$, it has been proved that every element in \( \mathrm{End}(\mathbb{A}^3_k) \) has a dynamical degree which is an algebraic integer of degree at most 6.
     
     \begin{theorem}~\cite[Corollary 3]{dang2021spectral}
         Dynamical degrees of polynomial automorphisms of $\mathbb{A}_k^3$ are algebraic numbers whose degree over $\mathbb{Q}$ is at most $6$.
     \end{theorem}
     
     Despite these results, very little is known in dimensions \( n \geq 4 \). Understanding the algebraic and arithmetic nature of dynamical degrees in higher dimensions remains a deep and largely open problem, and serves as the main motivation of the present work.
     \subsection{Main result}
     
     In this paper, we investigate the dynamical degrees of \( \mathrm{Aut}(\mathbb{A}^4_k) \), focusing on affine triangular automorphisms. While such maps have been well studied in dimension three—most notably in \cite{Bla}—they remain largely unexplored in dimension four. Our aim is to extend the study of this family of automorphisms to dimension four and to develop methods for determining their dynamical degrees explicitly. A key technique in the approach is the notion of \(\mu\)-algebraically stable (see Definition \ref{102}), a property introduced and utilized in the three-dimensional case to facilitate dynamical degree computations. Roughly speaking, a map is said to be \(\mu\)-algebraically stable if the sequence of degrees of its iterates exhibits multiplicative stability, and the degree of \( f^n \) grows exactly like \( \lambda(f)^n \), where \( \lambda(f) \) is the dynamical degree. When this property holds, the computation of \( \lambda(f) \) becomes possible.

     But it often happens that a polynomial automorphism is not \(\mu\)-algebraically stable, so in this paper, we adapt and generalize J.Blanc and I van santen’s framework to the four-dimensional setting. Since the dynamical degree is invariant under conjugation (Lemma \ref{008})—that is, for any two conjugate maps \( f \) and \( g \) in \( \mathrm{End}(\mathbb{A}^n_k) \), we have \( \lambda(f) = \lambda(g) \)—this invariance allows us to simplify the problem by reducing the study of a general affine-triangular automorphism in dimension four to a conjugate form that is either \(\mu\)-algebraically stable, or whose dynamical degree can be computed explicitly by direct methods. See Section \ref{103} for more details. 
     
    Our main result is Theorem \ref{016}.
	
	\begin{theorem}\label{016}
		Let $f\in \mathrm{Aut}(\mathbb{A}_k^4)$ be affine-triangular. Then $\lambda(f)$ is an algebraic integer of degree $\le 4$.
	\end{theorem}
	
	For $n\le 4$, if $\operatorname{char}(k)\neq 2$, every quadratic automorphism is conjugated by an affine automorphism to an affine-triangular 
	one, see \cite{meisters1991strong}. As a consequence of Theorem \ref{016}, we have
    
	\begin{corollary}\label{029}
		If $\operatorname{char}(k)\neq 2$, then dynamical degrees of quadratic automorphisms in dimension four are algebraic integers of degree $\le 4$.
	\end{corollary}

    We give the proof of Theorem \ref{016} in Section 3.
    Besides, we continue exploring the dynamical degrees of some specific polynomial automorphisms in Section 4. 
    Which the dynamical degrees are affirmative answers to Conjecture \ref{404}.
	
	\section{Preliminaries}
	In this section, we recall some basic definitions concerning polynomial automorphisms and recall the $\mu$-algebraically stable theory proposed by Blanc and
    van Santen \cite{blanc2022dynamical}.
    \subsection{$\mu$-algebraically stable}
	\begin{definition}\label{013} An automorphism $f=(f_1,\cdots ,f_n)\in \mathrm{Aut}(\mathbb{A}_k^n)$ is: 
		\begin{itemize}[leftmargin=*]
			\item triangular if $f_i=c_ix_i+p_i$, where $c_i\in k\setminus \left \{ 0 \right \} $ and $p_i\in k[x_1,\cdots ,x_{i-1}]$;
			\item elementary if $f_i=x_i$ for $i=1,\cdots,n-1$, and $f_n=c_nx_n+p_n$, where $c_n\in k\setminus \left \{ 0 \right \} $ and $p_n\in k\left [ x_1,\cdots,x_{n-1} \right ] $;
			\item affine if $\deg(f)=1$;
			\item permutation if $\{ f_1,\cdots,f_n \}=\{ x_1,\cdots,x_n \}$;
            \item affine-triangular if $f=\alpha\circ\tau$, where $\alpha$ is affine and $\tau$ is triangular;
            \item affine-elementary if $f=\alpha\circ\tau$, where $\alpha$ is affine and $\tau$ is elementary;
            \item permutation-triangular if $f=\alpha\circ\tau$, where $\alpha$ is permutation and $\tau$ is triangular;
            \item permutation-elementary if $f=\alpha\circ\tau$, where $\alpha$ is permutation and $\tau$ is elementary.
		\end{itemize}
	\end{definition}
	
	The dynamical degree is an invariant under conjugation.
	
	\begin{lemma}\label{008}
		Let $f,g \in \mathrm{Aut}(\mathbb{A}_k^n)$ and $h=g\circ f\circ g^{-1} \in \mathrm{Aut}(\mathbb{A}_k^n)$. Then $\lambda(h)=\lambda(f)$.
	\end{lemma}
	\begin{proof}
		Note that $h^r=g\circ f^r \circ g^{-1}$. 
        Then $\lambda(h)\le \lambda(f)$, since $\deg(h^r)\le \deg(g)\cdot\deg(f^r)\cdot\deg(g^{-1})$.  
        Similarly, $\lambda(f)\le \lambda(h)$ since $f^r=g^{-1}\circ h^r \circ g$.
	\end{proof}
	
	\begin{definition}
		For each $\mu =(\mu _1,\cdots,\mu _n)\in (\mathbb{R}_{\ge  0})^n\setminus \left \{ 0 \right \} $,
		the degree function is defined as
		$$\deg_\mu: k[x_1,\cdots,x_n]\setminus \left \{ 0 \right \}  \to \mathbb{R}_{\ge 0},$$
		$$\sum_{(a_1,\cdots,a_n)\in \mathbb{N}^n\setminus \left \{ 0 \right \}}c_{(a_1,\cdots,a_n)}x_1^{a_1}x_2^{a_2}\cdots x_n^{a_n}\mapsto \max\left \{ \sum_{i=1}^{n}a_i\mu _i\mid  c_{(a_1,\cdots,a_n)}\ne 0 \right \}.  $$
	\end{definition}
	\begin{definition}
		Let $\mu =(\mu _1,\cdots,\mu _n)\in (\mathbb{R}_{\ge  0})^n\setminus \left \{ 0 \right \}$.
		For each $f=(f_1,\cdots ,f_n)\in\mathrm{Aut} (\mathbb{A}_k^n )\setminus \left \{ 0 \right \}$,
		the $\mu$-degree of $f$ is defined by
		$$\deg_\mu (f)=\min \left \{ \theta \in \mathbb{R}_{\ge 0} \mid \deg_\mu(f_i)  \le \theta \mu _i \ \text {for each }\ i\in \left \{ 1,\cdots,n \right \}  \right \}, $$
		and $\deg_\mu(f)=\infty $ if the set is empty.
	\end{definition}
	\begin{definition}
		Let $f=(f_1,\cdots,f_n)\in \mathrm{Aut}(\mathbb{A}_k^n)$, 
		the square matrix $M=(m_{i,j})_{i,j=1}^{n}\in \mathrm{Mat}_n(\mathbb{N} ) $
		is said to be contained in $f$ if the coefficient of the monomial $\prod_{j=1}^{n}x_j^{m_{i,j}} $
		in $f_i$ is not zero.
		The set of matrices contained in $f$ is finite.
	\end{definition}
	\begin{definition}
		Let $f=(f_1,\cdots,f_n)\in \mathrm{Aut}(\mathbb{A}_k^n)$.
		\begin{enumerate}[leftmargin=*]
			\item 
			The maximal eigenvalue of $f$ is defined as
			$$\theta =\max\left \{ \left | \xi  \right | \in \mathbb{R} \mid \xi 
			\text{ is an eigenvalue of a matrix cotained in} \ f \right \} .$$ 
			\item 
			The maximal eigenvector of $f$ is defined as 
			$$\mu =\left \{(\mu_1,\cdots,\mu_n)\in \mathbb{R}_{\ge 0}^n 
            \mid \deg_\mu(f_i)=\theta \mu_i \ \text{for each }\ i\in \left \{ 1,\cdots,n \right \} \right \}.$$
		\end{enumerate}
	\end{definition}
	In particular, for a maximal eigenvector $\mu$, we get $\deg_\mu (f)=\theta < \infty $.
	
	\begin{theorem}\label{012}\cite[Proposition B]{blanc2022dynamical}
		Let $f=(f_1,\cdots,f_n)\in \mathrm{Aut}(\mathbb{A}_k^n)$ with maximal eigenvalue $\theta$.
		\begin{enumerate}[leftmargin=*]
			\item There is a maximal eigenvector 
			$\mu =(\mu_1,\cdots,\mu_n)\in (\mathbb{R}_{\ge 0})^n\setminus \left \{ 0 \right \}  $ of $f$.
			\item $1\le \lambda(f) \le \theta \le \deg(f) $. 
			If $\mu \in (\mathbb{R}_{> 0})^n $, then $\lambda(f)=\theta$.
		\end{enumerate}
	\end{theorem}
    \begin{definition}\label{102}
        If $\lambda(f)= \theta$, $f$ is said to be $\mu$-algebraically stable.
    \end{definition}
	
	\subsection{Preservation of a linear projection}
	\begin{definition}
		Let $f=(f_1,\cdots,f_n)\in \mathrm{Aut}(\mathbb{A}_k^n)$ and $ S=\left \{ i_1,\cdots,i_m \right \}$ 
		be a subset of $\left \{ 1,\cdots,n \right \} $
		such that $(f_{i_1},\cdots,f_{i_m})\in \mathrm{Aut}(\mathbb{A}_k^m) $, where each $f_{i_j}\in k[x_{i_1},\cdots,x_{i_m}]$. 
        Then we say $f$ has an $m$-linear projection $S$.
	\end{definition}

    \begin{remark} If $f=(f_1,\cdots,f_n)\in \mathrm{Aut}(\mathbb{A}_k^n)$ has an $m$-linear projection $S=\left \{ i_1,\cdots i_m \right \} $, 
    we may always assume that $S=\left \{ 1,\cdots ,m \right \} $ by relabeling variants. If $f=(f_1,\cdots,f_n)\in \mathrm{Aut}(\mathbb{A}_k^n)$ has linear projection $\left \{ 1,\cdots,m \right \} $, let $K=k(x_1,\cdots,x_m)$ be the fraction field of $k[x_1,\cdots,x_m]$. 
    Then we have $(f_{m+1},\cdots,f_n)\in \mathrm{Aut}(\mathbb{A}_K^{n-m})$ since $K[f_{m+1},\cdots,f_n]=K[x_{m+1},\cdots,x_n]$. 

	\end{remark}

    \begin{lemma}\label{001}\cite[Lemma 2.4.1]{blanc2022dynamical}
		Let $f=(f_1,\cdots,f_n)\in \mathrm{Aut}(\mathbb{A}_k^n)$ and suppose it has an $m$-linear projection $\left \{ 1,\cdots,m \right \}$.
		Let $\lambda _1=\lambda(f_1,\cdots,f_m)$, and let  $\lambda _2=\lim_{r \to \infty}\sqrt[r]{\deg_\mu(f^r)}$ where
		$\mu_1=\cdots=\mu_m=0$ and $\mu_{m+1}=\cdots=\mu_n=1$,
		then $\lambda (f)=\max\left \{ \lambda_1,\lambda_2 \right \}  $.
	\end{lemma}
    \begin{proposition}\label{027}
        If $f=(f_1,\cdots,f_n)\in \mathrm{Aut}(\mathbb{A}_k^n)$ has an $m$-linear projection $\left \{ 1,\cdots,m \right \} $ and 
		$g=id=(x_1,\cdots,x_m)\in \mathrm{Aut}(\mathbb{A}_k^m)$. Then $\lambda(f)=\lambda(f_{m+1},\cdots,f_n)$ and seen 
        $(f_{m+1},\cdots,f_n)$ as an automorphism in $\mathrm{Aut}(\mathbb{A}_K^{n-m})$.
    \end{proposition}
    \begin{proof}
        If $\mu=(\mu_1,\cdots,\mu_n)$ where $\mu_1=\cdots=\mu_m=0$ and $\mu_{m+1}=\cdots=\mu_n=1$, 
        we have $\lambda(f_{m+1},\cdots,f_n)=\lim_{r \to \infty}\sqrt[r]{\deg_\mu(f^r)}$. 
        The conclusion then follows from the Lemma \ref{001}.
    \end{proof}

	\section{Detailed description of main results}
    \subsection{Sketch}\label{103}
    
	Let $\varphi$ be an affine automorphism and $\delta$ be a triangular automorphism. 
	An affine-triangular automorphism of the form $f=\varphi \circ \delta$ can always 
	be conjugated to a permutation-triangular automorphism.
    \begin{lemma}\cite[Proposition 4.1.1]{blanc2022dynamical}
        Each affine-triangular automorphism of $\mathbb{A}_k^n$ is conjugated by an affine automorphism to a permutation-triangular automorphism.
    \end{lemma}
    Then we up to conjugation can rewrite all affine-triangular automorphisms as $f=\omega\circ s$ such that $\omega$ is a permutation and $s$ is triangular. 
    In dimension four, let $f=(f_1,f_2,f_3,f_4)$ and write $\omega:(x_1,x_2,x_3,x_4)\mapsto (x_{\omega(1)},x_{\omega(2)},x_{\omega(3)},x_{\omega(4)})$ or $(\omega(1)\ \omega(2)\  \omega(3)\ \omega(4))$ for simplicity. 
    \begin{proposition}\label{011}
		Let $f\in \mathrm{Aut}(\mathbb{A}_k^n)$ be an affine-triangular automorphism with $\left \{ 1,\cdots,n-1 \right \} $ linear projection, and 
		$f'=(f_1,\cdots,f_{n-1})$. Then $\lambda(f)=\lambda(f')$.
	\end{proposition}
	\begin{proof}
		It is obvious if $m=n-1$ by Lemma \ref{001}.
	\end{proof}
    \begin{lemma}\cite[Corollary 2.4.3]{blanc2022dynamical}\label{023}
        Let $f\in \mathrm{Aut}(\mathbb{A}_k^4)$ with 2-linear projection. 
        Then $\lambda(f)$ is an integer.
    \end{lemma}
    Below we improve the technique of simplification by conjugations. It is partially inspired by \cite[Lemma 2.6]{maubach2015invariants}.    
    \begin{proposition}\label{031}
    Let $f=(f_1,\cdots,f_n)$ be a permutation-triangular automorphism where $\omega=(\omega(1),\cdots,\omega(n))$. 
    If there is $i\in\left \{ 1,\cdots,n \right \} $. Then  
    \begin{align*}
    &f_i=c_{\omega(i)}x_{\omega(i)}+p_{\omega(i)}(x_1,\cdots,x_{\omega(i)-1})\\
    &f_{\omega^{-1}(i)}=c_ix_i+p_i(x_1,\cdots,x_{i-1})
    \end{align*}
    \begin{enumerate}[leftmargin=*]
        \item If $\omega(i)\neq i$ and if $\omega(i)>\omega(1),\cdots,\omega(i-1)$, then we can eliminate $p_{i}$ by conjugation.
        \item If $\omega(i)\neq i$ and for any $t\in \left \{ 1,2,\cdots,\omega(i)-1 \right \} $, there is 
        $j<i$ such that $\omega(j)=t$, then we can eliminate $p_{\omega(i)}$ by conjugation.
    \end{enumerate}
    \end{proposition}
    \begin{proof}
    We omit some unnecessary invariants for simplicity.
    \begin{enumerate}[leftmargin=*]
        \item Let $g=(x_1,\cdots,x_i-c_i^{-1}p_i(x_1,\cdots,x_{i-1}),\cdots,x_n)$. 
        
        Calculate $f'=g^{-1}\circ f\circ g$:
        $$f\circ g=(f_1',\cdots,f_{\omega^{-1}(i)}'=c_ix_i,f_{\omega^{-1}(i)+1'},\cdots,f_n')$$
        Then     
        $$f'=g^{-1}\circ f\circ g=(f_1',\cdots,f_i'+c_i^{-1}p_i(f_1,\cdots,f_{i-1}),\cdots,f_n')$$ 
        where $f'_u=f_u(x_1,\cdots,x_i-c_i^{-1}p_i(x_1,\cdots,x_{i-1}),\cdots,x_n)$ for each $u\in \mathbb{N}_{\le n}^+$. 
        Then $f'$ remains unchanged on the form of a permutation-triangular automorphism and it has no effect on $\omega$.
        
        \item Let $N=\left \{ j_1,\cdots,j_{\omega(i)-1}\right \}$ such that $\omega(j_r)=t_r\in \left \{ 1,2,\cdots,\omega(i)-1 \right \}$.
        Then $k[f_{j_1},\cdots,f_{j_{\omega(i)-1}}]=k[x_1,\cdots,x_{\omega(i)-1}]$.
        
        Hence there is $q(f_{j_1},\cdots,f_{j_{\omega(i)-1}})=p_{\omega(i)}(x_1,\cdots,x_{{\omega(i)}-1})$.
        
        Let $g=(x_1,\cdots,x_i+q(x_{j_1},\cdots,x_{j_{\omega(i)-1}}),\cdots,x_n)$. 
        
        Calculate $f'=g\circ f\circ g^{-1}$:
        $$
        f\circ g= (f_1(x_i+q(x_{j_1},\cdots,x_{j_{\omega(i)-1}})),\cdots,f_n(x_i+q(x_{j_1},\cdots,x_{j_{\omega(i)-1}})))
        $$
        Then
        $$f'=g^{-1}\circ f\circ g= (f'_1,\cdots,f'_{i}-q(f_{j_1},\cdots,f_{j_{\omega(i)-1}})=c_{\omega(i)}x_{\omega(i)},\cdots,f'_n)
        $$ 
        where $f'_u=f_u(x_1,\cdots,x_i+q(x_{j_1},\cdots,x_{j_{\omega(i)-1}}),\cdots,x_n)$ for each $u\in \mathbb{N}_{\le n}^+$.
        Then $f'$ remains unchanged on the form of a permutation-triangular automorphism and it has no effect on $\omega$.
    \end{enumerate}

    \end{proof}
    Now, we focus on permutation-triangular automorphisms in dimension $n=4$. 
    In total, there are only four possibilities by simplification:
        \begin{enumerate}[leftmargin=*]
            \item $f$ has the maximal eigenvalue $\theta=1$. Then $\lambda(f)=1$ since $1\le \lambda(f) \le \theta=1$ by Theorem \ref{012}.
            \item $f$ has a linear projection $\left \{ i,j \right \} $ where $i,j\in \left \{ 1,2,3,4\right \} $, then $\lambda(f)$ is an integer by Lemma \ref{023}.
            \item $f$ is $\mu$-algebraically stable.
            Then $\lambda(f)=\theta$ by Theorem \ref{012}, and $\theta$ is an algebraic integer of degree $\le4$ 
            since $\theta$ is an eigenvalue of a integer matrix $4\times 4$.
            \item $f$ has a linear projection $\left \{ i \right \} $ where $i\in \left \{ 1,2,3\right \} $. 
        \end{enumerate}
        When case (4) occurs, case (4) can be transferred to cases (1) (2) (3) which are shown in Subsection \ref{061} and Subsection \ref{062}.

	\subsection{Permutation $\omega$ induces linear projection}
    \subsubsection{3-linear projection $\left \{ 1,2,3\right \} $}
    
        There are 6 types such that the dynamical degrees of these types are equal to those of 
		affine-triangular automorphisms in dimension three by Proposition \ref{011}.
            $$\begin{Bmatrix}
               \omega_1=(1234)&\omega_2=(1324)&\omega_3=(2134)\\
               \omega_4=(2314)&\omega_5=(3124)&\omega_6=(3214) 
               \end{Bmatrix}
            $$
    \subsubsection{2-linear projection $\left \{ 1,2\right \} $}	
		
		There are 2 types. They are automorphisms with a 2-linear projection $\left \{ 1 ,2 \right \} $. Then the dynamical degrees of these types are algebraic integers according to the Lemma \ref{023}.
			$$\begin{Bmatrix}
			    \omega_7=(1243)&\omega_8=(2143)
			\end{Bmatrix}$$
    \subsubsection{1-linear projection $\left \{ 1\right \} $}\label{061}
    
        $$\begin{Bmatrix}
            \omega_{9}=(1342)&\omega_{10}=(1423)&\omega_{11}=(1432)
        \end{Bmatrix}$$
        
        We can assume $p_2=0$ since $2=\omega(i)>1=\omega(1)$ by Proposition \ref{031}.
        \begin{itemize}[leftmargin=*]
            \item $\omega_{9}=(1342)$: $f=(c_1x_1+p_1,c_3x_3+p_3,c_4x_4+p_4,c_2x_2)$
            
            Let $g=(x_2,x_1,x_3,x_4)$, then there are $h_i \in k[x_1,\cdots,x_{i-1}]$ such that
            $$f'=g\circ f'\circ g^{-1}=(c_3x_3+h_3,c_1x_2+p_1,c_4x_4+h_4,c_2x_1).$$

            Which is $\omega_{12}=(3241)$.

            \item $\omega_{10}=(1423)$: $f=(c_1x_1+p_1,c_4x_4+p_4,c_2x_2,c_3x_3+p_3)$

            We can assume $p_3=0$ since $i=4,\omega(4)=3,j_1=1,\omega(j_1)=1,j_2=3,\omega(j_2)=2$ by Proposition \ref{031}.
            
            Let $g=(x_1,x_3,x_2,x_4)$, then there is $l_4 \in k[x_1,x_2,x_3]$ such that 
            $$f'=g\circ f\circ g^{-1}=(c_1x_1+p_1,c_2x_3,c_4x_4+l_4,c_3x_2).$$
            Hence $f'\in \omega_{9}$. There is a representation of the relation of them by arrow graph below:
            $$\omega_{10}\xrightarrow{(x_1,x_3,x_2,x_4)} \omega_{9}\xrightarrow{(x_2,x_1,x_3,x_4)} \omega_{12}$$

            \item $\omega_{11}=(1432)$: $f=(c_1x_1+p_1,c_4x_4+p_4,c_3x_3+p_3,c_2x_2)$
           
            Let $g=(x_3,x_1,x_2,x_4)$, then there is $h_4 \in k[x_1,x_2,x_3]$ such that
            $$f'=g\circ f \circ g^{-1}=(c_4x_4+h_4,c_3x_1+p_3(x_2,x_3),c_1x_2+p_1,c_2x_3).$$
            
            Let $\theta$ be the maximal eigenvalue of $f'$, then there is a maximal eigenvector 
			$\mu =(\mu_1,\mu_2,\mu_3,\mu_4)\in \mathbb{R}_{\ge 0}^4\setminus \left \{ 0 \right \} $.
			By the definition of maximal eigenvector
            $$\begin{matrix}
                \mu_2=\mu_3\cdot\theta\\
                \mu_3=\mu_4\cdot\theta
            \end{matrix}$$
            Hence $\mu_2,\mu_3,\mu_4 $ are zero or none of them are zero.
            \begin{enumerate}
                \item If $\mu_2=\mu_3=\mu_4=0$, then $\mu_1\neq0$.It is a contradiction since $x_1\in f'_2$.
                \item If $\mu_2$, $\mu_3$, $\mu_4$ are not zero and $\mu_1=0$, it is a contradiction since $x_4\in f'_1$.
                \item If none of $\mu_1,\mu_2,\mu_3,\mu_4$ are zero, then $f'$ is $\mu$-algebraically stable and $\lambda(f')=\theta$.
            \end{enumerate}  
            
        \end{itemize}
	\subsection{Permutation $\omega$ fixes $f_2$ or $f_3$ but not $f_1$}
            $$\begin{Bmatrix}
                \omega_{12}=(3241)&\omega_{13}=(2431)&\omega_{14}=(4132)\\ \omega_{15}=(4213)&\omega_{16}=(4231)
            \end{Bmatrix}$$
        We can always assume $p_1=0$ by Proposition \ref{031} since $\omega $ does not fix $f_1$.
        \subsubsection{$\omega$ fixes only one of $f_2$ or $f_3$}\label{041}
        
        \begin{itemize}[leftmargin=*]
            \item $\omega_{12}=(3241)$ : $f=(c_3x_3,c_2x_2+p_2,c_4x_4+p_4,c_1x_1)$
            
                We can assume $p_3=0$ since $i=3$ and $\omega(3)=4>\omega(1)=3,\omega(2)=2$ by Proposition \ref{031}.
            \item $\omega_{13}=(4213)$ : $f=(c_4x_4+p_4,c_2x_2+p_2,c_1x_1,c_3x_3)$
            
                We can assume $p_3=0$ since $i=4,\omega(4)=3,j_1=2,\omega(j_1)=1,j_2=3,\omega(j_2)=2$ by Proposition \ref{031}.
            \item $\omega_{14}=(2431)$ : $f=(c_2x_2,c_4x_4+p_4,c_3x_3+p_3,c_1x_1)$
                
                We can assume $p_2=0$ since $i=2,\omega(2)=4>\omega(1)=2$ by Proposition \ref{031}.
            \item $\omega_{15}=(4132)$ : $f=(c_4x_4+p_4,c_1x_1,c_3x_3+p_3,c_2x_2)$
                
                We can assume $p_2=0$ since $i=4,\omega(4)=2,j_1=2,\omega(j_1)=1$ by Proposition \ref{031}.
            \end{itemize}

            Let $s\in \left \{2,3\right \} $ , then there are $r_1,r_2,r_3$ where $\left \{ s,r_1,r_2,r_3 \right \} =\left \{ 1,2,3,4 \right \} $ such that $\omega(r_1)=r_2,\omega(r_2)=r_3,\omega(r_3)=r_1$ 
            since $s$ is the only fixed element by permutation $\omega$. Assume $r_3=4$ and let $\theta$ be the maximal eigenvalue of $f$, then there is a maximal eigenvector 
				$\mu =(\mu_1,\mu_2,\mu_3,\mu_4)\in \mathbb{R}_{\ge 0}^4\setminus \left \{ 0 \right \} $.
				By the definition of maximal eigenvector
				$$\begin{matrix} 
					\mu_{r_1}=\mu_{r_2}\cdot \theta \\ \mu_{r_2}=\mu_{r_3}\cdot \theta 
				\end{matrix} $$
				Hence $\mu_{r_1},\mu_{r_2},\mu_{r_3} $ are zero or none of them are zero.
				\begin{enumerate}
					\item If $\mu_{r_1}=\mu_{r_2}=\mu_{r_3}=0 $ and $\mu_s\neq0$, then $\theta=1$ since $\mu_s=\theta \mu_s$.
					\item If $\mu_{r_1},\mu_{r_2},\mu_{r_3}\neq0 $ and $\mu_s=0$, then $p_s\in k$. Hence $f$ has $\left \{  s\right \}$ linear projection.
					\item If none of $\mu_1,\mu_2,\mu_3,\mu_4$ are zero, then $f$ is $\mu$-algebraically stable and $\lambda(f)=\theta$.    
				\end{enumerate}
            
        \subsubsection{$\omega$ fixes both $f_2$ and $f_3$}\label{081} 
                \begin{itemize}[leftmargin=*]
                
				\item $\omega_{16}=(4231)$ : $f=(c_4x_4+p_4,c_2x_2+p_2,c_3x_3+p_3,c_1x_1)$
                
				Let $\theta$ be the maximal eigenvalue of $f$, then there is a maximal eigenvector 
				$\mu =(\mu_1,\mu_2,\mu_3,\mu_4)\in \mathbb{R}_{\ge 0}^4\setminus \left \{ 0 \right \} $.
				By the definition of maximal eigenvector, $\mu_1=\mu_4\cdot \theta$.
				\begin{enumerate}
					\item If $\mu_1=\mu_4=0$,
					\begin{enumerate}
						\item If $\mu_2=0$, then $\mu_3\neq0$. Hence $\theta=1$ since $\mu_3=\theta \mu_3$. 
						Then $\lambda(f)=1$.
						\item If $\mu_3=0$, then $\mu_2\neq0$. Hence $\theta=1$ since $\mu_2=\theta \mu_2$. 
						Then $\lambda(f)=1$.
						\item If $\mu_2\neq0$ and $\mu_3\neq0$, then $p_4\in k\left [ x_1 \right ] $. Hence $f$ has $\left \{ 1, 4\right \}$ linear projection.
					\end{enumerate}
					\item If $\mu_1,\mu_4\neq0$ and either $\mu_2=0$ or $\mu_3=0$,
					\begin{enumerate}
						\item If $\mu_2=0$ and $\mu_3\neq0$, then $x_1\notin p_2$. Hence $f$ has $\left \{  2\right \}$ linear projection.
						\item If $\mu_3=0$ and $\mu_2\neq0$, then $x_1,x_2\notin p_3$. Hence $f$ has $\left \{  3\right \}$ linear projection.
						\item If $\mu_2=\mu_3=0$, then $x_1\notin p_2$ and $x_1\notin p_3$. Hence $f$ has $\left \{ 2 ,3\right \}$ linear projection.
					\end{enumerate}
					\item If none of $\mu_1,\mu_2,\mu_3,\mu_4$ are zero, then $f$ is $\mu$-algebraically stable and $\lambda(f)=\theta$.
				\end{enumerate}        
        \end{itemize}
        \subsubsection{Six cases in $\omega_{12}\sim \omega_{16}$}\label{062}
        There are 6 cases such that they have linear projection $\left \{ i \right \} $ where $i\in \left \{ 2,3 \right \} $ according to subsection \ref{041} an subsection \ref{081}. 
        Let $d\in k\setminus\left \{ 0  \right \} $ and $p_i\in k\left [ x_1,\cdots,x_{i-1} \right ]$, the cases are 
        \begin{enumerate}[leftmargin=*]
            \item $f_1=(c_3x_3,c_2x_2+d,c_4x_4+p_4,c_1x_1) \ \ \ (\text{from})\ \  \omega_{12}$
            \item $f_2=(c_4x_4+p_4,c_2x_2+d,c_1x_1,c_3x_3) \ \ \ (\text{from})\ \   \omega_{13}$
            \item $f_3=(c_2x_2,c_4x_4+p_4,c_3x_3+d,c_1x_1) \ \ \ (\text{from})\ \   \omega_{14}$
            \item $f_4=(c_4x_4+p_4,c_1x_1,c_3x_3+d,c_2x_2) \ \ \ (\text{from})\ \   \omega_{15}$
            \item $f_5=(c_4x_4+p_4,c_2x_2+d,c_3x_3+p_3,c_1x_1) \ \ \ (\text{from})\ \   \omega_{16}$
            \item $f_6=(c_4x_4+p_4,c_2x_2+p_2,c_3x_3+d,c_1x_1) \ \ \ (\text{from})\ \   \omega_{16}$
        \end{enumerate}
        Let $g_1=(x_2,x_1,x_3,x_4)$, $g_2=(x_1,x_3,x_2,x_4)$ and $g_3=(x_3,x_2,x_1,x_4)$, 
            $$\begin{matrix}
            g_1\circ f_3\circ g_1^{-1}=(c_1x_2,c_4x_4+q_4,c_3x_3+d,c_2x_1)\\
            g_2\circ f_2\circ g_2^{-1}=(c_2x_3,c_3x_2+d,c_4x_4+q_4,c_1x_1)\\
            g_3\circ f_4\circ g_3^{-1}=(c_1x_3,c_2x_2+d,c_4x_4+q_4,c_3x_1)
            \end{matrix}$$
        where $d\in k$ and $q_4\in k\left [ x_1,x_2,x_3 \right ]$. There is a representation of the relation of (1) (2) (3) (4) by arrow graph below:
        $$(4)\overset{g_1}{\longrightarrow}(3)\overset{g_2}{\longrightarrow}(1)\overset{g_3}{\longleftarrow}(2)$$
        And let $g_4=(x_3,x_1,x_2,x_4)$, then 
        $$g_4\circ f_1\circ g_4^{-1}=(c_2x_3+d,c_4x_4+q_4,c_3x_1,c_1x_2)$$
        where $d\in k$ and $q_4\in k\left [ x_1,x_2,x_3 \right ]$. It is a special case of type $\omega_{17}=(3412)$, there is a representation of the relation by arrow graph below:
        $$(1)\overset{g_2}{\hookrightarrow}\omega_{17}$$
        Similarly,
        $$g_2\circ f_6\circ g_2^{-1}=(c_4x_4+q_4,c_3x_2+d,c_2x_3+p_2,c_1x_1)$$
        And
        $$g_1\circ f_5\circ g_1^{-1}=(c_2x_1+d,c_4x_4+q_4,c_3x_3+p_3(x_2,x_1),c_1x_2)$$
        where $d\in k$ and $q_4\in k\left [ x_1,x_2,x_3 \right ]$. There is a representation of the relation of (5) (6) by arrow graph below:
        $$(6)\overset{g_2}{\hookrightarrow}(5)\overset{g_1}{\hookrightarrow}\omega_{11}.$$

    \subsection{Permutation $\omega$ has no fixed elements}    	
            $$\begin{Bmatrix}
                \omega_{17}=(3412)&\omega_{18}=(3421)&\omega_{19}=(4312)&\omega_{20}=(4123)\\
                \omega_{21}=(2341)&\omega_{22}=(3142)&\omega_{23}=(2413)&\omega_{24}=(4321)
            \end{Bmatrix}$$
            We can always assume $p_1=0$ by Proposition \ref{031} since $\omega $ does not fix $f_1$.
            \subsubsection{$p_2=0$ by simplification}	
			\begin{itemize}[leftmargin=*]
				
				\item $\omega_{17}=(3412)$ : $f=(c_3x_3+p_3,c_4x_4+p_4,c_1x_1,c_2x_2)$
                
				We can assume $p_2=0$ since $i=2,\omega(2)=4>\omega(1)=3$ by Proposition \ref{031}.
				Let $\theta$ be the maximal eigenvalue of $f$, then there is a maximal eigenvector 
				$\mu =(\mu_1,\mu_2,\mu_3,\mu_4)\in \mathbb{R}_{\ge 0}^4\setminus \left \{ 0 \right \} $.
				By the definition of maximal eigenvector
				$$\begin{matrix} 
					\mu_1=\mu_3\cdot \theta \\ \mu_2=\mu_4\cdot \theta 
				\end{matrix} $$
				Hence $\mu_1$ and $\mu_3$ are both zero or both not, $\mu_2$ and $\mu_4$ are both zero or both not.
				\begin{enumerate}
					\item If $\mu_1=\mu_3=0$ and $\mu_2,\mu_4 \neq 0$, then $p_3 \in k[x_1]$. 
					Hence $f$ has $\left \{ 1, 3\right \}$ linear projection.
					\item If $\mu_2=\mu_4=0$ and $\mu_1,\mu_3 \neq 0$, then $p_4 \in k[x_2]$. 
					Hence $f$ has $\left \{ 2, 4\right \}$ linear projection.
					\item If none of $\mu_1,\mu_2,\mu_3,\mu_4$ are zero, then $f$ is $\mu$-algebraically stable and $\lambda(f)=\theta$.
				\end{enumerate}
				
				\item $\omega_{18}=(3421)$ : $f=(c_3x_3+p_3,c_4x_4+p_4,c_2x_2,c_1x_1)$
				
				Let $g=(x_2,x_1,x_3,x_4)$, then there are $h_i \in k[x_1,\cdots,x_{i-1}]$ such that
				$$f'=g \circ f \circ g^{-1}=(c_4x_4+h_4,c_3x_3+h_3,c_2x_1,c_1x_2).$$
				
				This is one of the type $\omega_{19}=(4312)$.
				
				\item $\omega_{19}=(4312)$ : $f=(c_4x_4+p_4,c_3x_3+p_3,c_1x_1,c_2x_2)$
                
				We can assume $p_2=0$ since $i=4,\omega(4)=2,j_1=3,\omega(j_3)=1$ by Proposition \ref{031}.
				Let $\theta$ be the maximal eigenvalue of $f$, then there is a maximal eigenvector 
				$\mu =(\mu_1,\mu_2,\mu_3,\mu_4)\in \mathbb{R}_{\ge 0}^4\setminus \left \{ 0 \right \} $.
				By the definition of maximal eigenvector
				$$\begin{matrix} 
					\mu_1=\mu_3\cdot \theta \\ \mu_2=\mu_4\cdot \theta 
				\end{matrix} $$
				Hence $\mu_1$ and $\mu_3$ are both zero or both not, $\mu_2$ and $\mu_4$ are both zero or both not.
				\begin{enumerate}
					\item If $\mu_1=\mu_3=0$ and $\mu_2,\mu_4 \neq 0$. But it is a contradiction since $x_4\in f_1$.
					\item If $\mu_2=\mu_4=0$ and $\mu_1,\mu_3 \neq 0$. But it is a contradiction since $x_3\in f_2$.  
					\item If none of $\mu_1,\mu_2,\mu_3,\mu_4$ are zero, then $f$ is $\mu$-algebraically stable and $\lambda(f)=\theta$.
				\end{enumerate}
				
				\item $\omega_{20}=(4123)$ : $f=(c_4x_4+p_4,c_1x_1,c_2x_2,c_3x_3+p_3)$
                
				We can assume $p_2=0$ since $i=3,\omega(3)=2,j_1=2,\omega(j_1)=1$ by Proposition \ref{031}.
				Let $\theta$ be the maximal eigenvalue of $f$, then there is a maximal eigenvector 
				$\mu =(\mu_1,\mu_2,\mu_3,\mu_4)\in \mathbb{R}_{\ge 0}^4\setminus \left \{ 0 \right \} $.
				By the definition of maximal eigenvector
				$$\begin{matrix} 
					\mu_1=\mu_2\cdot \theta \\ \mu_2=\mu_3\cdot \theta 
				\end{matrix} $$
				Hence $\mu_1,\mu_2,\mu_3 $ are zero or none of them are zero.
				\begin{enumerate}
					\item If $\mu_1=\mu_2=\mu_3=0$ and $\mu_4\neq0$. But it is a contradiction since $x_4\in f_1$.
					\item If $\mu_1,\mu_2,\mu_3\neq0$ and $\mu_4=0$. But it is a contradiction, since $x_1,x_2,x_3\in f_4$.
					\item If none of $\mu_1,\mu_2,\mu_3,\mu_4$ are zero, then $f$ is $\mu$-algebraically stable and $\lambda(f)=\theta$.   
				\end{enumerate}
                \end{itemize}
        \subsubsection{$p_2=p_3=0$ by simplification} 
                \begin{itemize}[leftmargin=*]
			      \item $\omega_{21}=(2341)$ : $f=(c_2x_2,c_3x_3,c_4x_4+p_4,c_1x_1)$
                  
                We can assume $p_2=0$ since $i=2,\omega(2)=3>\omega(1)=2$ by Proposition \ref{031}.
                Then we can still assume $p_3=0$ since $i=3,\omega(3)=4>\omega(1)=2,\omega(2)=3$ by Proposition \ref{031}.

                Let $\theta$ be the maximal eigenvalue of $f$. Then there is a maximal eigenvector 
				$\mu =(\mu_1,\mu_2,\mu_3,\mu_4)\in \mathbb{R}_{\ge 0}^4\setminus \left \{ 0 \right \} $.
				By the definition of maximal eigenvector
                $$\begin{matrix}
                    \mu_2=\mu_1\cdot\theta\\
                    \mu_3=\mu_2\cdot\theta\\
                    \mu_1=\mu_4\cdot\theta
                \end{matrix}$$
                Then none of $\mu_1,\mu_2,\mu_3,\mu_4$ are zero. Hence $f$ is $\mu$-algebraically stable and $\lambda(f)=\theta$.

                \item $\omega_{22}=(3142)$ : $f=(c_3x_3,c_1x_1,c_4x_4+p_4,c_2x_2)$
				
				We can assume $p_3=0$ since $i=3,\omega(3)=4>\omega(1)=3,\omega(2)=1$ by Proposition \ref{031}.
				Then we can still assume $p_2=0$ since $i=4,\omega(4)=2,j_1=2,\omega(j_1)=1$ by Proposition \ref{031}. 
                \item $\omega_{23}=(2413)$ : $f=(c_2x_2,c_4x_4+p_4,c_1x_1,c_3x_3)$
				
				We can assume $p_2=0$ since $i=2,\omega(2)=4>\omega(1)=2$ by Proposition \ref{031}.
				Then we can still assume $p_3=0$ since $i=4,\omega(4)=3,j_1=3,\omega(j_1)=1,j_2=1,\omega(j_2)=2$ by Proposition \ref{031}.
                \end{itemize}

                Let $g_1=(x_2,x_1,x_3,x_4)$ and let $g_2=(x_1,x_3,x_2,x_4)$, there is the relation between 
                $\omega_{21} \sim \omega_{23} $:
                $$\begin{matrix}
                g_1\circ \omega_{21} \circ g_1^{-1}&\subseteq \omega_{22}&\Longleftrightarrow &g_1\circ \omega_{22} \circ g_1^{-1}&\subseteq\omega_{21}\\
                g_2\circ \omega_{22} \circ g_2^{-1}&\subseteq\omega_{23}&\Longleftrightarrow &g_2\circ \omega_{23} \circ g_2^{-1}&\subseteq\omega_{22}\\
                \end{matrix}$$
                since $g_1=g_1^{-1},g_2=g_2^{-1}$. This relation can be represented by the arrow graph:
                $$\omega_{21}\overset{g_1}{\rightleftharpoons } \omega_{22} 
                \overset{g_2}{\rightleftharpoons } \omega_{23}
                $$
                
	\subsubsection{No elimination by simplification}
                \begin{itemize}[leftmargin=*]
				\item $\omega_{24}=(4321)$ : $f=(c_4x_4+p_4,c_3x_3+p_3,c_2x_2+p_2,c_1x_1)$
				
				Let $\theta$ be the maximal eigenvalue of $f$, then there is a maximal eigenvector 
				$\mu =(\mu_1,\mu_2,\mu_3,\mu_4)\in \mathbb{R}_{\ge 0}^4\setminus \left \{ 0 \right \} $.
				By the definition of maximal eigenvector, $\mu_1=\mu_4\cdot \theta$.
				\begin{enumerate}
					\item If $\mu_1=\mu_4=0$,
					\begin{enumerate}
						\item If $\mu_2=0$ then $\mu_3\neq0$, then $f_3=c_2x_2+p_2\in k\left [ x_1,x_2 \right ] $. 
						It is a contradiction since we can get $\theta =0$.
						\item If $\mu_3=0$ then $\mu_2\neq0$. It is a contradiction since $x_2\in f_3$.
						\item If $\mu_2\neq0$ and $\mu_3\neq0$, then $p_4\in k\left [ x_1 \right ] $. Hence $f$ has $\left \{ 1, 4\right \}$ linear projection.
					\end{enumerate}
					\item If $\mu_1,\mu_4\neq0$,
					\begin{enumerate}
						\item If $\mu_2=0$ then $\mu_3=0$ since $x_3\in f_2$. Hence 
						$x_1 \notin p_2$ and $x_1 \notin p_3$ means that $f$ has $\left \{ 2 ,3\right \}$ linear projection.
						\item If $\mu_3=0$ then $\mu_2=0$ since $x_2\in f_3$. Hence 
						$x_1 \notin p_2$ and $x_1 \notin p_3$ means that $f$ has $\left \{ 2, 3\right \}$ linear projection.
						
					\end{enumerate}
					\item If none of $\mu_1,\mu_2,\mu_3,\mu_4$ are zero, then $f$ is $\mu$-algebraically stable and $\lambda(f)=\theta$.
				\end{enumerate}	     	    
			      \end{itemize}
        \begin{proof}[Proof of Theorem \ref{016}]
            The relation $\omega_1 \sim \omega_{24}$ completes the proof.
        \end{proof}

    \section{Two derivative results in higher dimensions}

    \subsection{Generalizations of shift-like automorphisms}
    For each $n\geq1$, a shift-like automorphism of $\mathbb{A}_k^n$ is of the form $(x_{n+1}+p(x_1,\cdots,x_n),x_1,\cdots,x_n)$ where $p\in k[x_1,\cdots,x_n]$.
    More information on the dynamics of such automorphisms has been studied in \cite{bera2018some}. The following lemma is the generalization of dynamical degrees of shift-like automorphisms.
    \begin{lemma}\cite[Proposition 4.2.3]{blanc2022dynamical}\label{007}
		Let $f$ be a permutation-elementary automorphism in dimension $n$, then $\lambda(f)$ is an algebraic integer of degree $\le n-1$.
	\end{lemma}
    
    \begin{corollary}
		Let $T_1,T_2\in \mathrm{Aut}(\mathbb{A}_k^n) $ be two permutation-triangular automorphisms such that 
		\begin{align*}
			T_1 & = (x_1,x_n+p_n,\cdots,x_3+p_3,x_2+p_2) \\
			T_2 & = (x_2+p_2,x_3+p_3,\cdots,x_n+p_n,x_1)
		\end{align*}
		where $p_i\in k[x_1,\cdots,x_{i-1}]$. Then $\lambda(T_1)$ and $\lambda(T_2)$ are algebraic integers of degree $\le n-1$.
	\end{corollary}
    \begin{proof}
		
		We first prove for $T_1$. Let $g_1=(x_1,x_2-p_2,x_3,\cdots,x_n)$, then $g_1^{-1}\circ T_1 \circ g_1$ is equal to
		$$
		\begin{bmatrix}
			x_1,\\
			x_n+p_n(x_1,x_2-p_2,\cdots)+p_2(x_1),\\
			x_{n-1}+p_{n-1}(x_1,x_2-p_2,\cdots),\\
			\vdots\\
			x_3+p_3(x_1,x_2-p_2,x_3),\\
			x_2\\
		\end{bmatrix}.
	    $$
		We rewrite it into 
		$$T_1'=(x_1,x_n+q_n,x_{n-1}+q_{n-1},\cdots,x_3+q_3,x_2)$$
		where $q_i\in k[x_1,\cdots,x_{i-1}]$.
		We can still assume $g_2=(x_1,x_2,x_3-q_3,\cdots)$ and calculate $g_2^{-1}\circ T_1'\circ g_2$ to eliminate $q_3$.
		After finite steps, it becomes a permutation-elementary automorphism and its dynamical degree is an algebraic integer 
		by Lemma \ref{007}. Then $\lambda(T_1)$ is an algebraic integer, since the dynamical degree is an invariant under conjugation by Lemma \ref{008}. 

        Now we prove for $T_2$.
        Let $g_1=(x_1,x_2-p_2,x_3,\cdots,x_n)$, then $g_1^{-1}\circ T_2 \circ g_1$ is equal to
		$$
		\begin{bmatrix}
			x_2,\\
			x_3+p_3(x_1,x_2-p_2,x_3)+p_2(x_2),\\
			\vdots\\
			x_{n-1}+p_{n-1}(x_1,x_2-p_2,\cdots),\\
			x_n+p_n(x_1,x_2-p_2,\cdots),\\
			x_1\\
		\end{bmatrix}.
		$$ 
		We rewrite it into 
		$$T_2'=(x_2,x_3+q_3,\cdots,x_{n-1}+q_{n-1},x_n+q_n,x_1)$$
		where $q_i\in k[x_1,\cdots,x_{i-1}]$.
		We can still assume $g_2=(x_1,x_2,x_3-q_3,\cdots)$ and calculate $g_2^{-1}\circ T_2'\circ g_2$ to eliminate $q_3$.
		After finite steps, it becomes a permutation-elementary automorphism and its dynamical degree is an algebraic integer 
		by Lemma \ref{007}. Then $\lambda(T_2)$ is an algebraic integer, since the dynamical degree is an invariant under conjugation by Lemma \ref{008}. 
        \end{proof}
    
    \subsection{Two automorphisms in dimension 5}
    In \cite{sun2014classification}, Sun classified the quadratic homogeneous automorphisms of $\mathrm{Aut}(\mathbb{A}_\mathbb{C}^5)$. 
    The quadratic homogeneous automorphisms are of the form $f_i=x_i+p_i$ where $p_i\in k[x_1,\cdots,x_{i-1},x_{i+1},\cdots,x_n]$ are quadratic homogeneous polynomials.
    There are only two possibilities which are not linearly triangularizable (cannot be conjugated into affine triangular forms).
    \begin{lemma}\cite[Theorem 3.2]{sun2014classification}
        Let $S_1,S_2\in\mathrm{Aut}(\mathbb{A}_\mathbb{C}^5)$ such that
        $$S_1=\begin{bmatrix}
		x_1,\\
		x_2+o_1x_1^2,\\
		x_3-x_1x_4+\tau _2(x_1,x_2),\\
		x_4+x_1x_5+x_2x_3+\tau_3(x_1,x_2),\\
		x_5+x_2x_4+\tau_4(x_1,x_2)\\
	\end{bmatrix}$$
        where $o_1\in\mathbb{C} $ and $\tau_2,\tau_3,\tau_4$ are quadratic homogeneous polynomials.
        $$S_2=  
            \begin{bmatrix} 
            x_1,\\
            x_2+ax_4^2+o_1x_1^2,\\
            x_3-2ax_4x_5+x_1x_2+o_2x_1^2,\\
            x_4+ax_5^2+x_1x_3+o_3x_1^2,\\
            x_5+x_1x_4+o_4x_1^2
            \end{bmatrix}$$
        where $a\neq0$ and $o_1,o_2,o_3,o_4\in \mathbb{C}$. Then $S_1$ and $S_2$ are not linearly triangularizable.
    \end{lemma}

    Now we give the description of $\lambda(S_1)$ and $\lambda(S_2)$.

    \begin{remark}\label{002}
		Assume that $f=(f_1,\cdots,f_n)\in \mathrm{Aut}(\mathbb{A}_k^n)$ has $m$-linear projection $\left \{ 1,\cdots,m \right \} $,  
		$g=(f_1,\cdots,f_m)\in \mathrm{Aut}(\mathbb{A}_k^m)$, then $g$ can also be seen as an automorphism of $\mathbb{A}_k^n$,
		which is $g=(f_1,\cdots,f_m,x_{m+1},\cdots,x_n)\in \mathrm{Aut}(\mathbb{A}_k^n)$.
		Write $F=(x_1,\cdots,x_m,f_{m+1},\cdots,f_n)\in \mathrm{Aut}(\mathbb{A}_k^n)$, $f=g\circ F$. Then 
		$$f^r = (g\circ F)^r= g^r\circ (g^{1-r}\circ F\circ g^{r-1})\circ \cdots \circ (g^{-1}\circ F\circ g)\circ F.$$
		Hence $f^r=g^r\circ G_r$ if we denote $G_r=(g^{1-r}\circ F\circ g^{r-1})\circ \cdots \circ (g^{-1}\circ F\circ g)\circ F$.
	\end{remark}

    \begin{proposition}\label{006}
		Assume that $f=(f_1,\cdots,f_n)\in \mathrm{Aut}(\mathbb{A}_k^n)$ has $m$-linear projection $\left \{ 1,\cdots,m \right \} $ and 
		$g=(f_1,\cdots,f_m)\in \mathrm{Aut}(\mathbb{A}_k^m)$. 
		Then $\lambda(f)=\lambda(g)$ if $(f_{m+1},\cdots,f_n)\in \mathrm{Aut}(\mathbb{A}_K^{n-m})  $ is affine.
	\end{proposition}
	\begin{proof}
		We can write $f^r=g^r\circ G_r$ by Remark \ref{002}. Then $G_r$ is affine and $\mathrm{deg_\mu}(f^r)=1 $ since the compositions of the affine automorphisms are still affine.
		
		Let $\mu =(0,\cdots,0,1,\cdots,1)$ where $\mu_m=0$ and $\mu_{m+1}=\cdots=1$, then $\lim_{r \to \infty } \sqrt[r]{\mathrm{deg_\mu}(f^r)}=1$.
		Then $\lambda(f)=\mathrm{max}\left \{ \lambda(g), 1 \right \} =\lambda(g)$ by Lemma \ref{001}.
	\end{proof}

    \begin{corollary}\label{028}
		$\lambda(S_1)=1$.
	\end{corollary}
	\begin{proof}
		Let $S=(f_1,f_2,f_3,f_4,f_5)$ and let $K=k(x_1,x_2)$ be the fraction field of $k[x_1,x_2]$. Then 
		$(f_3,f_4,f_5)\in \mathrm{Aut}(\mathbb{A}_K^3) $ is affine and $g=(x_1,x_2+o_1x_1^2)$. Hence $\lambda(S)=1$ by Proposition \ref{006}.
	\end{proof}

    \begin{proposition}
        $\lambda(S_2)$ is an algebraic integer of degree $\le 4$.
    \end{proposition}
    \begin{proof}
        It is a direct consequence of Corollary \ref{029} and Proposition \ref{027} when $m=1$ and $n=5$.
    \end{proof}
        
	\bibliographystyle{alpha}
	\bibliography{Refshao.bib}
\end{document}